\documentclass[a4paper,12pt]{article}
\usepackage[top=2.5cm,bottom=2.5cm,left=2.5cm,right=2.5cm]{geometry}
\usepackage{cite, amsmath, amssymb}
\usepackage{amssymb}
\usepackage{graphicx}
\newenvironment{proof}{{\noindent \it Proof.}}{\hfill $\blacksquare$\par}
\usepackage{latexsym}
\usepackage{graphicx,booktabs,multirow}
\usepackage{tikz}
\usetikzlibrary{decorations.pathreplacing}
\usetikzlibrary{intersections}
\usepackage{enumerate}
\usepackage{graphicx,booktabs,multirow}
\usepackage{appendix}
\newtheorem{theorem}{Theorem}[section]

\newtheorem{lemma}[theorem]{Lemma}

\begin{document}

\title{Extremal cacti with respect to Sombor index}
\author{Hechao Liu\thanks{Corresponding author}
 \\
{\small School of Mathematical Sciences, South China Normal University,}\\ {\small Guangzhou, 510631, P. R. China}\\
 \small {\tt hechaoliu@m.scnu.edu.cn}\\\
\small { (Received August 21, 2021)}
}
\date{}
\maketitle
\begin{abstract}
Recently, a novel topological index, Sombor index, was introduced by Gutman, defined as $SO(G)=\sum\limits_{uv\in E(G)}\sqrt{d_{u}^{2}+d_{v}^{2}}$, where $d_{u}$ denotes the degree of vertex $u$.

In this paper, we first determine the maximum Sombor index among cacti with $n$ vertices and $t$ cycles, then determine the maximum Sombor index among cacti with perfect matchings. We also characterize corresponding maximum cacti.
\end{abstract}

\hskip0.2cm

\maketitle
\maketitle

\makeatletter
\renewcommand\@makefnmark%
{\mbox{\textsuperscript{\normalfont\@thefnmark)}}}
\makeatother

 \baselineskip=0.30in

\section{Introduction}

In this paper, all notations and terminologies can refer to Bondy and Murty \cite{jaus2008}.

The Sombor index and reduced Sombor index are defined as \cite{gumn2021}
$$SO(G)=\sum_{uv\in E(G)}\sqrt{d_{u}^{2}+d_{v}^{2}},$$
$$SO_{red}(G)=\sum_{uv\in E(G)}\sqrt{(d_{u}-1)^{2}+(d_{v}-1)^{2}}.$$

Immediately after, R. Cruz et al.\cite{ctra2021} studied the extremal Sombor index among unicyclic graphs and bicyclic graphs. The same author also \cite{rirm2021} determined the extremal Sombor index of chemical graphs. At the same time, using different methods, Deng et al. \cite{dengt2021} also determined molecular trees with extremal values of Sombor indices. Wang et al. \cite{wmlf2021} obtain the relations between Sombor and other degree-based indices. Liu et al. \cite{lyhu2021} ordered the chemical graphs by their Sombor index. Red\v{z}epovi\'{c} \cite{redz2021} studied chemical applicability of Sombor indices.
Other results can be found in \cite{algh2021,chli2021,dxsa2021,fyli2021,guma2021,huli2021,hoxu2021,milo2021,lyfh2021,tliu2021,trdo2021,zhou2021,zylh2021}.

Cacti is a graph that any two cycles have at most one common vertex. Denote by $\mathcal{H}(n,t)$, $\mathcal{C}(2\beta,t)$, the collection of cacti with $n$ vertices and $t$ cycles, the collection of cacti with perfect matchings, $2\beta$ vertices and $t$ cycles, respectively.
There are a lot of research about topological indices of graphs with perfect matching or given matching numbers, we can refer to \cite{duzu2010,lidt19,ling20,made2011,zhou2021} and references cited therein.

In this paper, we first determine the maximum Sombor index among cacti $\mathcal{H}(n,t)$, then determine the maximum Sombor index among cacti $\mathcal{C}(2\beta,t)$. We also characterize corresponding maximum cacti.

\section{Preliminaries}

In the following, we give a few important lemmas which will be useful in the main results.

\begin{figure}[ht!]
  \centering
  \scalebox{.18}[.18]{\includegraphics{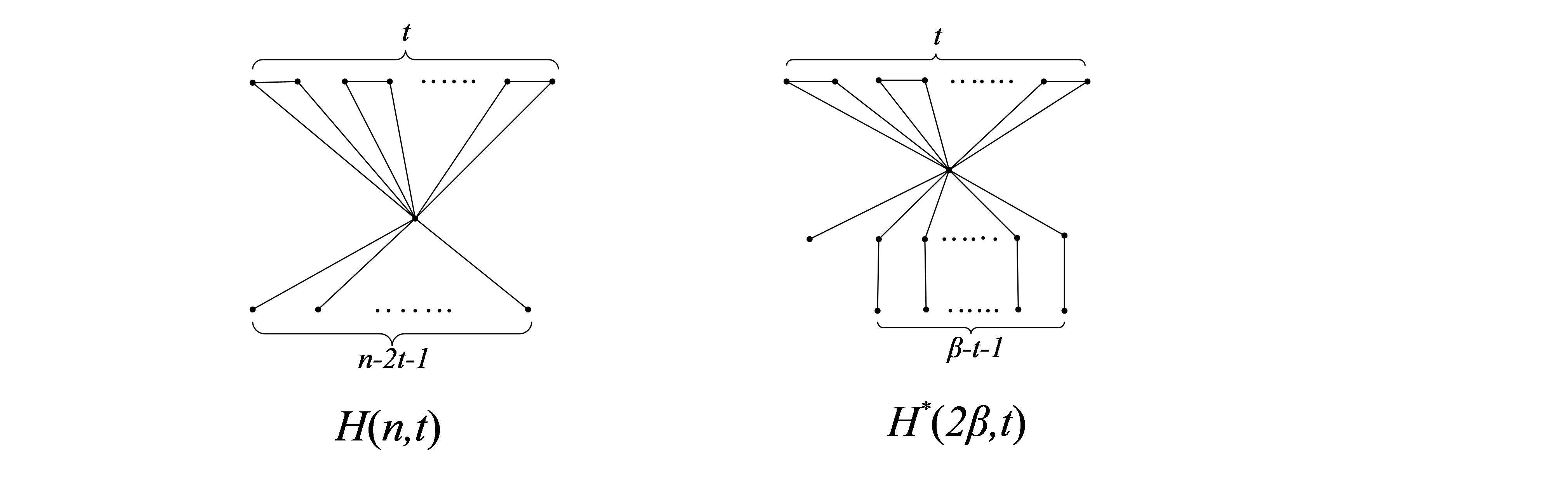}}
  \caption{Cacti $H(n,t)$ and $H^{*}(2\beta,t)$.}
 \label{fig-1}
\end{figure}

\begin{lemma}\label{l-21}\cite{gumn2021}
Let $G\in \mathcal{H}(n,0)$, then $SO(G)\leq (n-1)\sqrt{(n-1)^{2}+1}$, with equality iff $G\cong H(n,0)$.
\end{lemma}

\begin{lemma}\label{l-22}\cite{ctra2021}
Let $G\in \mathcal{H}(n,1)$, then $SO(G)\leq (n-3)\sqrt{(n-1)^{2}+1}+2\sqrt{(n-1)^{2}+4}+2\sqrt{2}$, with equality iff $G\cong H(n,1)$.
\end{lemma}

\begin{lemma}\label{l-23}
Let $f_{1}(x)=\sqrt{x^{2}+d^{2}}-\sqrt{x^{2}+(d-r)^{2}}$, $d,r$ are constants, $0<r\leq d$. Then $f_{1}(x)$ is a monotonically decreasing function.
\end{lemma}
\begin{proof}
$f_{1}^{'}(x)=\frac{x}{\sqrt{x^{2}+d^{2}}}-\frac{x}{\sqrt{x^{2}+(d-r)^{2}}}<0$. Thus $f_{1}(x)$ is a monotonically decreasing function.
\end{proof}

\begin{lemma}\label{l-24}
Let $f_{2}(x)=r\sqrt{x^{2}+1^{2}}+(x-r)\sqrt{x^{2}+2^{2}}+(x-r)\sqrt{(x-r)^{2}+2^{2}}$ $($$r$ are constants$)$ and $x\geq r$, then $f_{2}(x)$ is a monotonically increasing function.
\end{lemma}
\begin{proof}
$f_{2}^{'}(x)=\frac{rx}{\sqrt{x^{2}+1}}+\sqrt{x^{2}+4}+\frac{x(x-r)}{\sqrt{x^{2}+4}}+\sqrt{(x-r)^{2}+4}+\frac{(x-r)^{2}}{\sqrt{(x-r)^{2}+4}}>0$. Thus $f_{2}(x)$ is a monotonically increasing function.
\end{proof}

\begin{lemma}\label{l-25}
Let $f_{3}(x)=x\sqrt{x^{2}+2^{2}}-(x-2)\sqrt{(x-2)^{2}+2^{2}}$, then $f_{3}(x)$ is a monotonically increasing function.
\end{lemma}
\begin{proof}
$f_{3}^{'}(x)=\sqrt{x^{2}+2^{2}}-\sqrt{(x-2)^{2}+2^{2}}+\frac{x^{2}}{\sqrt{x^{2}+2^{2}}}-\frac{(x-2)^{2}}{\sqrt{(x-2)^{2}+2^{2}}}>0$. Thus $f_{3}(x)$ is a monotonically increasing function.
\end{proof}

\begin{lemma}\label{l-26}\cite{chli2021,zhou2021}
Let $G\in \mathcal{C}(2\beta,0)$, then $SO(G)\leq \sqrt{5}(\beta-1)+\sqrt{\beta^{2}+1}+(\beta-1)\sqrt{\beta^{2}+4}$, with equality iff $G\cong H^{*}(2\beta,0)$.
\end{lemma}

\begin{lemma}\label{l-27}\cite{zhou2021}
Let $G\in \mathcal{C}(2\beta,1)$, then $SO(G)\leq \beta\sqrt{(\beta+1)^{2}+4}+\sqrt{(\beta+1)^{2}+1}+\sqrt{5}(\beta-1)+2\sqrt{2}$, with equality iff $G\cong H^{*}(2\beta,1)$.
\end{lemma}

\begin{lemma}\label{l-28}
Let $f(x)=(x-1)\sqrt{x^{2}+2^{2}}+\sqrt{x^{2}+1^{2}}$ and $g(x)=f(x)-f(x+1)$, then $g(x)$ $(x\geq 1)$ is a monotonically decreasing function.
\end{lemma}
\begin{proof}
$f^{'}(x)=\sqrt{x^{2}+2^{2}}+\frac{x(x-1)}{\sqrt{x^{2}+2^{2}}}+\frac{x}{\sqrt{x^{2}+1^{2}}}$;
$f^{''}(x)=\frac{x}{\sqrt{x^{2}+2^{2}}}+ \frac{2x-1}{\sqrt{x^{2}+2^{2}}}-\frac{x^{2}(x-1)}{(x^{2}+2^{2})^{\frac{3}{2}}}+ \frac{1}{\sqrt{x^{2}+1^{2}}}- \frac{x^{2}}{(x^{2}+1^{2})^{\frac{3}{2}}}=\frac{2x(x^{2}+6)-4}{(x^{2}+4)^{\frac{3}{2}}}+\frac{1}{(x^{2}+1)^{\frac{3}{2}}}>0 $ for $x\geq 1$.
Thus $g(x)$ $(x\geq 1)$ is a monotonically decreasing function.
\end{proof}

\ \notag\

\section{Maximum Sombor index among cacti $\mathcal{H}(n,t)$}

\begin{theorem}\label{t-31}
Let $G\in \mathcal{H}(n,t)$ $($$n\geq 5$$)$, then
$$SO(G)\leq (n-2t-1)\sqrt{(n-1)^{2}+1}+2t\sqrt{(n-1)^{2}+4}+2\sqrt{2}t,$$
with equality iff $G\cong H(n,t)$.
\end{theorem}
\begin{proof}
For convenience, we denote $Q(n,t)\triangleq (n-2t-1)\sqrt{(n-1)^{2}+1}+2t\sqrt{(n-1)^{2}+4}+2\sqrt{2}t$.
In the following we make inductive assumptions about $n+t$.

By Lemma \ref{l-21} and \ref{l-22}, the conclusion holds if $t=1$ or $t=2$.
If $n=5$, the conclusion holds clearly. So we only consider $n\geq 6$ and $t\geq 2$ in the following.
$PV(G)$ denotes the set of pendant vertices in $G$. We call the vertices connected with pendant vertices support vertices, denoted by $Supp(G)$.

\noindent {\bf Case 1}. $PV\neq \emptyset$.

Let $v$ be a pendant vertex, $N(v)=w$, $N(w)=\{v,x_{1},x_{2},\cdots,x_{d-1}\}$. $d_{x_{i}}=1$ for $1\leq i\leq r-1$, $d_{x_{i}}\geq 2$ for $r\leq i\leq d-1$.

Let $G^{*}=G-v-x_{1}-x_{2}-\cdots -x_{r-1}$, then $G^{*}\in \mathcal{H}(n-r,t)$. By Lemma \ref{l-23}, \ref{l-24} and \ref{l-25}, we have

\begin{equation}
\begin{split}
 SO(G)&=SO(G^{*})+r\sqrt{d^{2}+1^{2}}+\sum_{i=r}^{d-1}(\sqrt{d^{2}+d_{x_{i}}^{2}}-\sqrt{(d-r)^{2}+d_{x_{i}}^{2}})\\
&\leq  Q(n-r,t)+r\sqrt{d^{2}+1^{2}}+\sum_{i=r}^{d-1}(\sqrt{d^{2}+d_{x_{i}}^{2}}-\sqrt{(d-r)^{2}+d_{x_{i}}^{2}})\\
&\leq  Q(n,t)+2t\sqrt{(n-r-1)^{2}+2^{2}}-2t\sqrt{(n-1)^{2}+2^{2}}+(n-r-2t-1)\sqrt{(n-r-1)^{2}+1^{2}} \\
&\quad -(n-2t-1)\sqrt{(n-1)^{2}+1^{2}}+r\sqrt{d^{2}+1^{2}}+(d-r)(\sqrt{d^{2}+2^{2}}-\sqrt{(d-r)^{2}+2^{2}}) \\
&\leq  Q(n,t)+2t\sqrt{(n-r-1)^{2}+2^{2}}-2t\sqrt{(n-1)^{2}+2^{2}}+(n-r-2t-1)\sqrt{(n-r-1)^{2}+1^{2}} \\
&\quad -(n-2t-1)\sqrt{(n-1)^{2}+1^{2}}+r\sqrt{(n-1)^{2}+1^{2}}\\
&\quad +(n-r-1)(\sqrt{(n-1)^{2}+2^{2}}-\sqrt{(n-r-1)^{2}+2^{2}}) \\
&= Q(n,t)+ (n-2t-r-1)[(\sqrt{(n-1)^{2}+2^{2}}-\sqrt{(n-r-1)^{2}+2^{2}})\\
&\quad -(\sqrt{(n-1)^{2}+1^{2}}-\sqrt{(n-r-1)^{2}+1^{2}})]\\
&\leq Q(n,t),  \nonumber
\end{split}
\end{equation}
with equality iff $G\cong H(n,t)$.

\noindent {\bf Case 2}. $PV= \emptyset$.

Suppose that there exists vertices $v_{0}$, $v_{1}$, $v_{2}$ on a cycle of $G$, $v_{0}v_{1}, v_{1}v_{2}\in E(G)$, $d(v_{1})=d(v_{2})=2$, $d(v_{0})\triangleq d\geq 3$. In the following, we classify the circle lengths.

\noindent {\bf Subcase 2.1}. The circle length greater than or equal to 4.

Let $G^{*}=G-v_{1}+v_{0}v_{2}$, then $G^{*}\in \mathcal{H}(n-1,t)$.
\begin{equation}
\begin{split}
 SO(G)&=SO(G^{*})+\sqrt{d^{2}+2^{2}}+\sqrt{2^{2}+2^{2}}-\sqrt{d^{2}+2^{2}}\\
&\leq  Q(n-1,t)+2\sqrt{2}\\
&=  Q(n,t)+2t\sqrt{(n-2)^{2}+2^{2}}-2t\sqrt{(n-1)^{2}+2^{2}}+(n-2t-2)\sqrt{(n-2)^{2}+1^{2}} \\
&\quad -(n-2t-1)\sqrt{(n-1)^{2}+1^{2}}+2\sqrt{2} \\
&=  Q(n,t)+2t[\sqrt{(n-2)^{2}+2^{2}}-\sqrt{(n-1)^{2}+2^{2}}]+(n-2t-2)[\sqrt{(n-2)^{2}+1^{2}} \\
&\quad -\sqrt{(n-1)^{2}+1^{2}}]-\sqrt{(n-1)^{2}+1^{2}}+2\sqrt{2} \\
&< Q(n,t).  \nonumber
\end{split}
\end{equation}

\noindent {\bf Subcase 2.2}. The circle length equals 3.

Let $G^{*}=G-v_{1}-v_{2}$, then $G^{*}\in \mathcal{H}(n-2,t-1)$. By Lemma \ref{l-23}, \ref{l-25}, we have
\begin{equation}
\begin{split}
 SO(G)&=SO(G^{*})+2\sqrt{d^{2}+2^{2}}+\sqrt{2^{2}+2^{2}}+\sum_{i=1}^{d-2}(\sqrt{d^{2}+d_{x_{i}}^{2}}-\sqrt{(d-2)^{2}+d_{x_{i}}^{2}})\\
&\leq  Q(n-2,t-1)+2\sqrt{d^{2}+2^{2}}+\sqrt{2^{2}+2^{2}}+(d-2)(\sqrt{d^{2}+2^{2}}-\sqrt{(d-2)^{2}+2^{2}})\\
&=  Q(n,t)+(2t-2)\sqrt{(n-3)^{2}+2^{2}}-2t\sqrt{(n-1)^{2}+2^{2}}+(n-2t-1)\sqrt{(n-3)^{2}+1^{2}} \\
&\quad -(n-2t-1)\sqrt{(n-1)^{2}+1^{2}}+2\sqrt{d^{2}+2^{2}}+(d-2)(\sqrt{d^{2}+2^{2}}-\sqrt{(d-2)^{2}+2^{2}}) \\
&\leq  Q(n,t)+(2t-2)\sqrt{(n-3)^{2}+2^{2}}-2t\sqrt{(n-1)^{2}+2^{2}}+(n-2t-1)\sqrt{(n-3)^{2}+1^{2}} \\
&\quad -(n-2t-1)\sqrt{(n-1)^{2}+1^{2}}+(n-1)\sqrt{(n-1)^{2}+2^{2}}-(n-3)\sqrt{(n-3)^{2}+2^{2}} \\
&=  Q(n,t)+(n-2t-1)[(\sqrt{(n-1)^{2}+2^{2}}-\sqrt{(n-3)^{2}+2^{2}})\\
&\quad -(\sqrt{(n-1)^{2}+1^{2}}-\sqrt{(n-3)^{2}+1^{2}})] \\
&\leq Q(n,t),  \nonumber
\end{split}
\end{equation}
with equality iff $G\cong H(n,t)$.

This completes the proof.
\end{proof}

Using a similar way, for the reduced Sombor index, we also have similar result. We
omit the proof.
\begin{theorem}\label{t-32}
Let $G\in \mathcal{H}(n,t)$ $($$n\geq 5$$)$, then
$$SO_{red}(G)\leq (n-2t-1)(n-2)+2t\sqrt{(n-2)^{2}+1}+\sqrt{2}t,$$
with equality iff $G\cong H(n,t)$.
\end{theorem}

\section{Maximum Sombor index among cacti $\mathcal{C}(2\beta,t)$}

For convenience, we denote
$\Phi(\beta,t)\triangleq SO(H^{*}(2\beta,t))=(\beta+t-1)\sqrt{(\beta+t)^{2}+4}+\sqrt{(\beta+t)^{2}+1}+\sqrt{5}(\beta-t-1)+2\sqrt{2}t$.
Note that the definition of function $g(x)$ and $f(x)$ has introduced in Lemma \ref{l-28}.
Suppose $\mathcal{M}$ is a perfect matching of $G\in \mathcal{C}(2\beta,t)$.

In the following Lemma \ref{l-41}, Lemma \ref{l-42} and Lemma \ref{l-43}, we make inductive assumptions about $\beta+t$. By Lemma \ref{l-26} and \ref{l-27}, these conclusions hold if $t=1$ or $t=0$.
If $\beta=3, t=2$, the conclusion holds clearly. So we only consider $\beta\geq 4$, $t\geq 2$ in the following.
\begin{lemma}\label{l-41}
Let $G\in \mathcal{C}(2\beta,t)$ $($$\beta\geq 2$$)$, $\delta(G)\geq 2$, then
$SO(G)< \Phi(\beta,t)$.
\end{lemma}
\begin{proof}
Suppose the cycle $C=v_{1}v_{2}\cdots v_{\lambda}v_{1}$ where $3\leq d_{v_{1}}\triangleq d\leq \beta+t$, $d_{v_{i}}=2$ for $i=2,3,\cdots,\lambda$.
$N(v_{1})=\{v_{2},v_{\lambda},x_{1},x_{2},\cdots,x_{d-2}\}$.
Let $G^{*}=G-v_{1}v_{2}$, then $G^{*}\in \mathcal{C}(2\beta,t-1)$, and $SO(G^{*})\leq \Phi(\beta,t-1)$.
By Lemma \ref{l-23} and \ref{l-28}, we have
\begin{equation}
\begin{split}
 SO(G)&=SO(G^{*})+\sqrt{d^{2}+2^{2}}+(\sqrt{2^{2}+2^{2}}-\sqrt{2^{2}+1^{2}})+\sum_{i=1}^{d-1}(\sqrt{d^{2}+d_{x_{i}}^{2}}-\sqrt{(d-1)^{2}+d_{x_{i}}^{2}})\\
&\leq  \Phi(\beta,t)+(\beta+t-2)\sqrt{(\beta+t-1)^{2}+2^{2}}+\sqrt{(\beta+t-1)^{2}+1^{2}} \\
&\quad -(\beta+t-1)\sqrt{(\beta+t)^{2}+2^{2}}-\sqrt{(\beta+t)^{2}+1^{2}}+t\sqrt{t^{2}+2^{2}}-(t-1)\sqrt{(t-1)^{2}+2^{2}} \\
&=  \Phi(\beta,t)+g(\beta+t-1)-g(d-1)+[(\sqrt{t^{2}+2^{2}}-\sqrt{(t-1)^{2}+2^{2}})\\
&\quad -(\sqrt{t^{2}+1^{2}}-\sqrt{(t-1)^{2}+1^{2}})] \\
&< \Phi(\beta,t),  \nonumber
\end{split}
\end{equation}
This completes the proof.
\end{proof}
\ \notag\

The support vertices are the vertices connected with pendant vertices.
\begin{lemma}\label{l-42}
Let $G\in \mathcal{C}(2\beta,t)$ $($$\beta\geq 2$$)$, $\delta(G)=1$ and there exists a support vertex $u$ with degree $2$ of the corresponding pendant vertex $v$, then
$SO(G)\leq \Phi(\beta,t)$,
with equality iff $G\cong H^{*}(2\beta,t)$.
\end{lemma}
\begin{proof}
Let $N(u)=\{v,w\}$, $N(w)=\{u,x_{1},x_{2},\cdots,x_{r},x_{r+1},\cdots,x_{d-1}\}$. $d_{x_{i}}=1$ for $1\leq i\leq r$, $d_{x_{i}}\geq 2$ for $r+1\leq i\leq d-1$.
For convienience, denote $d_{w}=d$.
Let $G^{*}=G-v-u$, then $G^{*}\in \mathcal{C}(2\beta-2,t)$. By Lemma \ref{l-23} and \ref{l-28}, we have
\begin{equation}
\begin{split}
SO(G)&=SO(G^{*})+\sum_{i=r+1}^{d-1}(\sqrt{d^{2}+d_{x_{i}}^{2}}-\sqrt{(d-1)^{2}+d_{x_{i}}^{2}})\\
&\quad +r(\sqrt{d^{2}+1^{2}}-\sqrt{(d-1)^{2}+1^{2}})+\sqrt{d^{2}+2^{2}}+\sqrt{2^{2}+1^{2}}\\
&\leq  \Phi(\beta-1,t)+(d-r-1)(\sqrt{d^{2}+2^{2}}-\sqrt{(d-1)^{2}+2^{2}})\\
&\quad +r(\sqrt{d^{2}+1^{2}}-\sqrt{(d-1)^{2}+1^{2}})+\sqrt{d^{2}+2^{2}}+\sqrt{2^{2}+1^{2}}\\
&=  \Phi(\beta,t)+(\beta+t-2)\sqrt{(\beta+t-1)^{2}+2^{2}}+\sqrt{(\beta+t-1)^{2}+1^{2}} \\
&\quad -(\beta+t-1)\sqrt{(\beta+t)^{2}+2^{2}}-\sqrt{(\beta+t)^{2}+1^{2}}\\
&\quad +(d-r-1)(\sqrt{d^{2}+2^{2}}-\sqrt{(d-1)^{2}+2^{2}})\\
&\quad +r(\sqrt{d^{2}+1^{2}}-\sqrt{(d-1)^{2}+1^{2}})+\sqrt{d^{2}+2^{2}}  \\
&=  \Phi(\beta,t)+g(\beta+t-1)+r(\sqrt{d^{2}+1^{2}}-\sqrt{(d-1)^{2}+1^{2}}) \\
&\quad +(d-r-1)(\sqrt{d^{2}+2^{2}}-\sqrt{(d-1)^{2}+2^{2}})+\sqrt{d^{2}+2^{2}}.  \nonumber
\end{split}
\end{equation}
Since $G\in \mathcal{C}(2\beta,t)$, then $r\leq 1$. We consider the following two cases.

\noindent {\bf Case 1}. $r=0$.
\begin{equation}
\begin{split}
SO(G)&\leq \Phi(\beta,t)+g(\beta+t-1)+(d-1)(\sqrt{d^{2}+2^{2}}-\sqrt{(d-1)^{2}+2^{2}})+\sqrt{d^{2}+2^{2}}\\
&=  \Phi(\beta,t)+g(\beta+t-1)-g(d-1)+[(\sqrt{d^{2}+2^{2}}-\sqrt{(d-1)^{2}+2^{2}}) \\
&\quad -(\sqrt{d^{2}+1^{2}}-\sqrt{(d-1)^{2}+1^{2}})]\\
&<\Phi(\beta,t).  \nonumber
\end{split}
\end{equation}

\noindent {\bf Case 2}. $r=1$.
\begin{equation}
\begin{split}
SO(G)&\leq \Phi(\beta,t)+g(\beta+t-1)+(\sqrt{d^{2}+1^{2}}-\sqrt{(d-1)^{2}+1^{2}})\\
&\quad +(d-2)(\sqrt{d^{2}+2^{2}}-\sqrt{(d-1)^{2}+2^{2}})+\sqrt{d^{2}+2^{2}}\\
&=  \Phi(\beta,t)+g(\beta+t-1)-g(d-1) \\
&\leq \Phi(\beta,t),  \nonumber
\end{split}
\end{equation}
with equality iff $G\cong H^{*}(2\beta,t)$.

This completes the proof.
\end{proof}
\ \notag\

\begin{lemma}\label{l-43}
Let $G\in \mathcal{C}(2\beta,t)$ $($$\beta\geq 2$$)$, $\delta(G)=1$ and the degrees of all support vertices are at least $3$, then
$SO(G)\leq \Phi(\beta,t)$,
with equality iff $G\cong H^{*}(2\beta,t)$.
\end{lemma}
\begin{proof}
Suppose that $C=v_{1}v_{2}\cdots v_{\lambda}v_{1}$ is such a cycle that $v_{i}(2\leq i\leq \lambda)$ does not on other cycles of $G$,
i.e., $d_{v_{i}}=2$ or $3$ for $2\leq i\leq \lambda$. $3\leq d_{v_{1}}\triangleq d\leq \beta+t$.
Without loss of generality, suppose $v_{1}v_{2}\notin \mathcal{M}$.
In the following, we classify the circle lengths.

\noindent {\bf Case 1}. The circle length equals 3.

Let $N(v_{1})=\{v_{2},v_{3},x_{1},x_{2},\cdots,x_{r},x_{r+1},\cdots,x_{d-2}\}$ where $r=0$ or $1$, $d_{x_{i}}=1$ for $1\leq i\leq r$, $d_{x_{i}}\geq 2$ for $r+1\leq i\leq d-2$.

\noindent {\bf Subcase 1.1}. $d_{v_{2}}=d_{v_{3}}=2$.

Let $G^{*}=G-v_{2}-v_{3}$, then $G^{*}\in \mathcal{C}(2\beta-2,t-1)$, $SO(G^{*})\leq \Phi(\beta-1,t-1)$.
\begin{equation}
\begin{split}
SO(G)&=SO(G^{*})+\sum_{i=r+1}^{d-2}(\sqrt{d^{2}+d_{x_{i}}^{2}}-\sqrt{(d-2)^{2}+d_{x_{i}}^{2}})+r(\sqrt{d^{2}+1^{2}}-\sqrt{(d-2)^{2}+1^{2}})\\
&\quad +2\sqrt{d^{2}+2^{2}}+\sqrt{2^{2}+2^{2}}\\
&\leq  \Phi(\beta,t)+(\beta+t-3)\sqrt{(\beta+t-2)^{2}+2^{2}}+\sqrt{(\beta+t-2)^{2}+1^{2}} \\
&\quad -(\beta+t-1)\sqrt{(\beta+t)^{2}+2^{2}}-\sqrt{(\beta+t)^{2}+1^{2}}+\sum_{i=r+1}^{d-2}(\sqrt{d^{2}+d_{x_{i}}^{2}}-\sqrt{(d-2)^{2}+d_{x_{i}}^{2}}) \\
&\quad +r(\sqrt{d^{2}+1^{2}}-\sqrt{(d-2)^{2}+1^{2}})+2\sqrt{d^{2}+2^{2}}. \nonumber
\end{split}
\end{equation}
\noindent {\bf Subcase 1.11}. $r=0$.

By Lemma \ref{l-23} and \ref{l-28}, we have
\begin{equation}
\begin{split}
SO(G)&\leq \Phi(\beta,t)+(\beta+t-3)\sqrt{(\beta+t-2)^{2}+2^{2}}+\sqrt{(\beta+t-2)^{2}+1^{2}} \\
&\quad -(\beta+t-1)\sqrt{(\beta+t)^{2}+2^{2}}-\sqrt{(\beta+t)^{2}+1^{2}} \\
&\quad +(d-2)(\sqrt{d^{2}+2^{2}}-\sqrt{(d-2)^{2}+2^{2}})+2\sqrt{d^{2}+2^{2}}\\
&=  \Phi(\beta,t)+[f(\beta+t-2)-f(\beta+t)]-[f(d-2)-f(d)] \\
&\quad +[(\sqrt{d^{2}+2^{2}}-\sqrt{(d-2)^{2}+2^{2}})-(\sqrt{d^{2}+1^{2}}-\sqrt{(d-2)^{2}+1^{2}})] \\
& < \Phi(\beta,t)+[g(\beta+t-2)+g(\beta+t-1)]-[g(d-2)+g(d-1)] \\
&=   \Phi(\beta,t)+[g(\beta+t-2)-g(d-2)]+[g(\beta+t-1)-g(d-1)]\\
&\leq \Phi(\beta,t). \nonumber
\end{split}
\end{equation}
\noindent {\bf Subcase 1.12}. $r=1$.

By Lemma \ref{l-23} and \ref{l-28}, we have
\begin{equation}
\begin{split}
SO(G)&\leq \Phi(\beta,t)+(\beta+t-3)\sqrt{(\beta+t-2)^{2}+2^{2}}+\sqrt{(\beta+t-2)^{2}+1^{2}} \\
&\quad -(\beta+t-1)\sqrt{(\beta+t)^{2}+2^{2}}-\sqrt{(\beta+t)^{2}+1^{2}}+(\sqrt{d^{2}+1^{2}}-\sqrt{(d-2)^{2}+1^{2}}) \\
&\quad +(d-3)(\sqrt{d^{2}+2^{2}}-\sqrt{(d-2)^{2}+2^{2}})+2\sqrt{d^{2}+2^{2}}\\
&=  \Phi(\beta,t)+[f(\beta+t-2)-f(\beta+t)]-[f(d-2)-f(d)] \\
&= \Phi(\beta,t)+[g(\beta+t-2)+g(\beta+t-1)]-[g(d-2)+g(d-1)] \\
&=   \Phi(\beta,t)+[g(\beta+t-2)-g(d-2)]+[g(\beta+t-1)-g(d-1)]\\
&\leq \Phi(\beta,t), \nonumber
\end{split}
\end{equation}
with equality iff $G\cong H^{*}(2\beta,t)$.

\noindent {\bf Subcase 1.2}. $d_{v_{2}}=d_{v_{3}}=3$.

Let $N(v_{2})=\{v_{1},v_{3},v_{2}^{'}\}$ and $N(v_{3})=\{v_{1},v_{2},v_{3}^{'}\}$.
Let $G^{*}=G-v_{2}^{'}-v_{3}^{'}$, then $G^{*}\in \mathcal{C}(2\beta-2,t)$, $SO(G^{*})\leq \Phi(\beta-1,t)$.
Note that $\beta\geq 4$, $t\geq 2$.
\begin{equation}
\begin{split}
SO(G)&=SO(G^{*})+2(\sqrt{d^{2}+3^{2}}-\sqrt{d^{2}+2^{2}})+\sqrt{3^{2}+3^{2}}+2\sqrt{3^{2}+1^{2}}-\sqrt{2^{2}+2^{2}}\\
&\leq  \Phi(\beta-1,t)+2(\sqrt{d^{2}+3^{2}}-\sqrt{d^{2}+2^{2}})+2\sqrt{10}+\sqrt{2}\\
&=  \Phi(\beta,t)+g(\beta+t-1)+2(\sqrt{d^{2}+3^{2}}-\sqrt{d^{2}+2^{2}})-\sqrt{5}+2\sqrt{10}+\sqrt{2}  \\
&\leq  \Phi(\beta,t)+g(5)+2(\sqrt{3^{2}+3^{2}}-\sqrt{3^{2}+2^{2}})-\sqrt{5}+2\sqrt{10}+\sqrt{2}  \\
&=  \Phi(\beta,t)+4\sqrt{5^{2}+2^{2}}+\sqrt{5^{2}+1^{2}}-5\sqrt{6^{2}+2^{2}}-\sqrt{6^{2}+1^{2}}\\
&\quad +2(\sqrt{3^{2}+3^{2}}-\sqrt{3^{2}+2^{2}})-\sqrt{5}+2\sqrt{10}+\sqrt{2}  \\
&< \Phi(\beta,t). \nonumber
\end{split}
\end{equation}

\noindent {\bf Subcase 1.3}. $d_{v_{2}}=2$, $d_{v_{3}}=3$.

Let $N(v_{3})=\{v_{1},v_{2},v_{3}^{'}\}$ and $N(v_{1})=\{v_{2},v_{3},x_{1},x_{2},\cdots,x_{d-2}\}$.
Let $G^{*}=G-v_{3}-v_{3}^{'}$, then $G^{*}\in \mathcal{C}(2\beta-2,t-1)$, $SO(G^{*})\leq \Phi(\beta-1,t-1)$.
\begin{equation}
\begin{split}
SO(G)&=SO(G^{*})+\sum_{i=1}^{d-2}(\sqrt{d^{2}+d_{x_{i}}^{2}}-\sqrt{(d-1)^{2}+d_{x_{i}}^{2}})+\sqrt{d^{2}+3^{2}}+\sqrt{d^{2}+2^{2}}\\
&\quad +\sqrt{2^{2}+3^{2}}+\sqrt{1^{2}+3^{2}}-\sqrt{(d-1)^{2}+1^{2}}\\
&\leq  \Phi(\beta-1,t-1)+(d-2)(\sqrt{d^{2}+2^{2}}-\sqrt{(d-1)^{2}+2^{2}})+\sqrt{d^{2}+3^{2}}+\sqrt{d^{2}+2^{2}}\\
&\quad +\sqrt{2^{2}+3^{2}}+\sqrt{1^{2}+3^{2}}-\sqrt{(d-1)^{2}+1^{2}}\\
&\leq  \Phi(\beta,t)+(\beta+t-3)\sqrt{(\beta+t-2)^{2}+2^{2}}+\sqrt{(\beta+t-2)^{2}+1^{2}} \\
&\quad -(\beta+t-1)\sqrt{(\beta+t)^{2}+2^{2}}-\sqrt{(\beta+t)^{2}+1^{2}}+(d-2)(\sqrt{d^{2}+2^{2}}-\sqrt{(d-1)^{2}+2^{2}})\\
&\quad +\sqrt{d^{2}+3^{2}}+\sqrt{d^{2}+2^{2}}+\sqrt{13}+\sqrt{10}-2\sqrt{2}-\sqrt{(d-1)^{2}+1^{2}}\\
&=  \Phi(\beta,t)+[f(\beta+t-2)-f(\beta+t)]+(d-2)(\sqrt{d^{2}+2^{2}}-\sqrt{(d-1)^{2}+2^{2}}) \\
&\quad +\sqrt{d^{2}+3^{2}}+\sqrt{d^{2}+2^{2}}+\sqrt{13}+\sqrt{10}-2\sqrt{2}-\sqrt{(d-1)^{2}+1^{2}}\\
&=  \Phi(\beta,t)+[f(\beta+t-2)-f(\beta+t)]-[f(d-1)-f(d)] \\
&\quad +\sqrt{d^{2}+3^{2}}-\sqrt{d^{2}+1^{2}}+\sqrt{13}+\sqrt{10}-2\sqrt{2}\\
&=  \Phi(\beta,t)+[f(\beta+t-1)-f(\beta+t)]-[f(d-1)-f(d)]+[f(\beta+t-2)-f(\beta+t-1)] \\
&\quad +\sqrt{d^{2}+3^{2}}-\sqrt{d^{2}+1^{2}}+\sqrt{13}+\sqrt{10}-2\sqrt{2}\\
&=  \Phi(\beta,t)+[g(\beta+t-1)-g(d-1)]+g(\beta+t-2) \\
&\quad +\sqrt{d^{2}+3^{2}}-\sqrt{d^{2}+1^{2}}+\sqrt{13}+\sqrt{10}-2\sqrt{2}\\
&\leq  \Phi(\beta,t)+g(\beta+t-2)+\sqrt{d^{2}+3^{2}}-\sqrt{d^{2}+1^{2}}+\sqrt{13}+\sqrt{10}-2\sqrt{2}. \nonumber
\end{split}
\end{equation}

If $d\geq 3$, then
\begin{equation}
\begin{split}
SO(G)&\leq  \Phi(\beta,t)+g(4)+\sqrt{3^{2}+3^{2}}-\sqrt{3^{2}+1^{2}}+\sqrt{13}+\sqrt{10}-2\sqrt{2}\\
&=  \Phi(\beta,t)+3\sqrt{4^{2}+2^{2}}+\sqrt{4^{2}+1^{2}}-4\sqrt{5^{2}+2^{2}}-\sqrt{5^{2}+1^{2}} \\
&\quad +\sqrt{3^{2}+3^{2}}-\sqrt{3^{2}+1^{2}}+\sqrt{13}+\sqrt{10}-2\sqrt{2}\\
&<  \Phi(\beta,t). \nonumber
\end{split}
\end{equation}

If $d=2$, then $\beta=2$ and $t=1$. Thus $G\cong H^{*}(4,1)$.

\noindent {\bf Case 2}. The circle length greater than or equal to 4.

\noindent {\bf Subcase 2.1}. $d_{v_{2}}=3$.

Let $N(v_{2})=\{v_{1},v_{3},v_{2}^{'}\}$, then $d_{v_{3}}=2$ or $3$.

Let $G^{*}=G-v_{2}-v_{2}^{'}+v_{1}v_{3}$, then $G^{*}\in \mathcal{C}(2\beta-2,t)$, $SO(G^{*})\leq \Phi(\beta-1,t)$.
\begin{equation}
\begin{split}
 SO(G)&=SO(G^{*})+\sqrt{d^{2}+3^{2}}+\sqrt{3^{2}+1^{2}}+\sqrt{3^{2}+d_{v_{3}}^{2}}-\sqrt{d^{2}+d_{v_{3}}^{2}}\\
&\leq  \Phi(\beta-1,t)+\sqrt{d^{2}+3^{2}}+\sqrt{3^{2}+1^{2}}+\sqrt{3^{2}+d_{v_{3}}^{2}}-\sqrt{d^{2}+d_{v_{3}}^{2}}\\
&=   \Phi(\beta,t)+(\beta+t-2)\sqrt{(\beta+t-1)^{2}+2^{2}}+\sqrt{(\beta+t-1)^{2}+1^{2}} \\
&\quad -(\beta+t-1)\sqrt{(\beta+t)^{2}+2^{2}}-\sqrt{(\beta+t)^{2}+1^{2}}-\sqrt{5} \\
&\quad +\sqrt{d^{2}+3^{2}}+\sqrt{3^{2}+1^{2}}+\sqrt{3^{2}+d_{v_{3}}^{2}}-\sqrt{d^{2}+d_{v_{3}}^{2}}. \nonumber
\end{split}
\end{equation}

\noindent {\bf Subcase 2.11}. $d_{v_{3}}=2$.
\begin{equation}
\begin{split}
 SO(G)&\leq   \Phi(\beta,t)+(\beta+t-2)\sqrt{(\beta+t-1)^{2}+2^{2}}+\sqrt{(\beta+t-1)^{2}+1^{2}} \\
&\quad -(\beta+t-1)\sqrt{(\beta+t)^{2}+2^{2}}-\sqrt{(\beta+t)^{2}+1^{2}}-\sqrt{5} \\
&\quad +\sqrt{d^{2}+3^{2}}-\sqrt{d^{2}+2^{2}}+\sqrt{10}+\sqrt{13}\\
&\leq  \Phi(\beta,t)+g(\beta+t-1)+3\sqrt{2}+\sqrt{10}-\sqrt{5}\\
&\leq  \Phi(\beta,t)+g(5)+3\sqrt{2}+\sqrt{10}-\sqrt{5}\\
&=     \Phi(\beta,t)+4\sqrt{29}+\sqrt{26}-5\sqrt{40}-\sqrt{37}+3\sqrt{2}+\sqrt{10}-\sqrt{5}\\
&<  \Phi(\beta,t). \nonumber
\end{split}
\end{equation}

\noindent {\bf Subcase 2.12}. $d_{v_{3}}=3$.
\begin{equation}
\begin{split}
 SO(G)&\leq   \Phi(\beta,t)+(\beta+t-2)\sqrt{(\beta+t-1)^{2}+2^{2}}+\sqrt{(\beta+t-1)^{2}+1^{2}} \\
&\quad -(\beta+t-1)\sqrt{(\beta+t)^{2}+2^{2}}-\sqrt{(\beta+t)^{2}+1^{2}}-\sqrt{5}+\sqrt{10}+3\sqrt{2}\\
&\leq  \Phi(\beta,t)+g(\beta+t-1)+3\sqrt{2}+\sqrt{10}-\sqrt{5}\\
&\leq  \Phi(\beta,t)+g(5)+3\sqrt{2}+\sqrt{10}-\sqrt{5}\\
&=     \Phi(\beta,t)+4\sqrt{29}+\sqrt{26}-5\sqrt{40}-\sqrt{37}+3\sqrt{2}+\sqrt{10}-\sqrt{5}\\
&<  \Phi(\beta,t). \nonumber
\end{split}
\end{equation}

\noindent {\bf Subcase 2.2}. $d_{v_{2}}=2$.

Since $v_{1}v_{2}\notin \mathcal{M}$, then $v_{2}v_{3}\in \mathcal{M}$, so $d_{v_{3}}=2$.

Let $G^{*}=G-v_{3}v_{4}+v_{2}v_{4}$, then $G^{*}\in \mathcal{C}(2\beta,t)$.
\begin{equation}
\begin{split}
 SO(G)&=SO(G^{*})+\sqrt{d^{2}+2^{2}}-\sqrt{d^{2}+3^{2}}+\sqrt{2^{2}+2^{2}}+\sqrt{d_{v_{4}}^{2}+2^{2}}-\sqrt{1^{2}+3^{2}}-\sqrt{d_{v_{4}}^{2}+3^{2}}\\
&\leq  SO(G^{*})+\sqrt{d^{2}+2^{2}}-\sqrt{d^{2}+3^{2}}+4\sqrt{2}-\sqrt{10}-\sqrt{13}\\
&<  SO(G^{*})+4\sqrt{2}-\sqrt{10}-\sqrt{13}\\
&<  SO(G^{*}). \nonumber
\end{split}
\end{equation}
When $\lambda-1=3$, by Case 1, we know $SO(G^{*})\leq \Phi(\beta,t)$; When $\lambda-1\geq 4$, by Case 2.1, we know $SO(G^{*})< \Phi(\beta,t)$. Thus we have
$SO(G)< \Phi(\beta,t)$.

This completes the proof.
\end{proof}
\ \notag\

By Lemma \ref{l-41}, \ref{l-42} and \ref{l-43}, we can obtain the maximum Sombor index among cacti $\mathcal{C}(2\beta,t)$.
\begin{theorem}\label{t-44}
Let $G\in \mathcal{C}(2\beta,t)$ $($$\beta\geq 2$$)$, then
$SO(G)\leq \Phi(\beta,t)$,
with equality iff $G\cong H^{*}(2\beta,t)$.
\end{theorem}

Using a similar way, for the reduced Sombor index, we also have similar result. We omit the proof.
\begin{theorem}\label{t-45}
Let $G\in \mathcal{C}(2\beta,t)$ $($$\beta\geq 2$$)$, then
$SO_{red}(G)\leq SO_{red}(H^{*}(2\beta,t))$,
with equality iff $G\cong H^{*}(2\beta,t)$.
\end{theorem}

\end{document}